\documentclass{amsart}
\usepackage{amsfonts,amssymb,amsmath,amsthm}
\usepackage{url}
\usepackage{enumerate}
\newtheorem{thm}{Theorem}[section]
\newtheorem{cor}[thm]{Corollary}
\newtheorem{lem}[thm]{Lemma}
\newtheorem{prop}[thm]{Proposition}
\theoremstyle{definition}

\theoremstyle{remark}

\numberwithin{equation}{section}

\begin{document}
\title[Dynamic of abelian subgroups of $Aut(\mathbb{C}^{n})$]{On the dynamic
of\\
 holomorphic  diffeomorphisms  groups \\ fixing  a commune
point, on $\mathbb{C}^{n}$}

\author{Yahya N'dao and Adlene Ayadi}

 \address{Yahya N'dao, University of Moncton, Department of mathematics and statistics, Canada}
 \email{yahiandao@yahoo.fr }
\address{Adlene Ayadi, University of Gafsa, Faculty of sciences, Department of Mathematics,Gafsa, Tunisia.}

\email{adlenesoo@yahoo.com}

\thanks{This work is supported by the research unit: syst\`emes dynamiques et combinatoire:
99UR15-15} 

\begin{abstract}
In this paper, we study the action on $\mathbb{C}^{n}$ of any
group $G$ of holomorphic diffeomorphisms (automorphisms) of
$\mathbb{C}^{n}$ fixing $0$. Suppose that there is $x\in
\mathbb{C}^{n}$,  having an orbit
 which generates $\mathbb{C}^{n}$ and also $\widetilde{E}(x)=\mathbb{C}^{n}$, where
 $\widetilde{E}(x)$ is the vector space generated by
$L_{G}=\{D_{0}fx,\ f\in G\}$. We give an important condition so
that an orbit $G(x)$ is isomorphic (by linear map) to the orbit
$L_{G}(x)$ of the linear group $L_{G}$. More if $G$ is abelian, we
prove the existence of a $G$-invariant open set $U$, dense in
$\mathbb{C}^{n}$, in which every orbit $O$ is relatively minimal
(i.e. $\overline{O}\cap U$ is a closed non empty, $G$-invariant
set and has no proper subset with these properties). Moreover, if
$G$ has a dense orbit in $\mathbb{C}^{n}$ then every orbit of $U$
is dense in $\mathbb{C}^{n}$.
\end{abstract}
\maketitle

\section{\bf Introduction }

By definition, diffeomorphisms of $\mathbb{C}^{n}$ which are
holomorphic are called automorphisms. Denote by
$Aut(\mathbb{C}^{n})$ the group of all automorphisms of
$\mathbb{C}^{n}$, it is very large and complicated for $n\geq 2$.
One attempt to study it, is to study its finite subgroups or more
general its subgroups which are isomorphic to compact Lie groups.
This paper is among the few who studied  the topological dynamic
of abelian  infinite subgroups of $Aut(\mathbb{C}^{n})$. Let $G$
be a subgroup of $Aut(\mathbb{C}^{n})$ fixing $0$. We consider the
action $G\times \mathbb{C}^{n} \longrightarrow \mathbb{C}^{n}$
given by $(f, x) \longmapsto f(x)$ (i.e the image of $x$ by $f$).
For a point $x\in\mathbb{C}^{n}$, denote by $G(x) =\{f(x), \ f\in
G\}\subset \mathbb{C}^{n}$ the orbit of $G$ through $x$. A subset
$E\subset\mathbb{C}^{n}$  is called $G$-invariant if $f(E)\subset
E$ for any $f\in G$; that is $E$ is a union of orbits. Denote by
$\overline{E}$ (resp. $\overset{\circ}{E}$ ) the closure (resp.
interior) of $E$.\medskip

A subset $E$ of  $\mathbb{C}^{n}$   is called \emph{a minimal set}
of  $G$  if  $E$ is closed in $ \mathbb{C}^{n}$,  non empty,
$G$-invariant and has no proper subset with these properties. It
is equivalent to say that  $E$  is a\ $G$-invariant set such that
every orbit contained in $E$ is dense in it. If  $\Omega$ is a
$G$-invariant set in  $ \mathbb{C}^{n}$,  we say that  $E$  is  a
\emph{minimal set in \ $\Omega$} if it is a minimal set of the
restriction  $G_{/\Omega}$  of $G$  to  $\Omega$. An orbit
$O\subset \Omega$  is called \emph{relatively minimal in $\Omega$}
if
 $ \overline{O}\cap \Omega$  is a minimal set in  $\Omega$.
  This means that for every $x\in \overline{O}\cap\Omega$ we have
$\overline{O}\cap\Omega=\overline{G(x)}\cap\Omega$. For example, a
closed orbit in  $\Omega$  is relatively minimal in $\Omega$. In
particular, every point in  $Fix (G)$  is relatively minimal in
$\mathbb{C}^{n}$.

 \medskip

In \cite{EAL}, the authors give a comprehensive introduction to
holomorphic dynamics, that is the dynamics induced by the
iteration of various analytic maps in complex number spaces. In
~\cite{EA2} the first named author studied the subgroup
$Aut_{1}(\mathbb{C}^{n})$ consisting of automorphisms with Jaccobi
determinant 1 (''volume-preserving automorphisms'') and in
~\cite{EALL}, the authors proved that for $n\geq 2$ shears
generate a dense subgroup of $Aut_{1}(\mathbb{C}^{n})$, but this
subgroup, at least $n=2$, is not the entire group
$Aut_{1}(\mathbb{C}^{n})$.

For a complex hyperbolic manifold $M$, denote by $Aut(M)$ be the
group of holomorphic automorphisms of $M$. Proper actions by
holomorphic transformations are found in abundance. A fundamental
result due to Kaup (see \cite{WK}) states that every closed
subgroup of $Aut(M)$ that preserves a continuous distance on M
acts properly on M. Thus, Lie groups acting properly and
effectively on M by holomorphic transformations are precisely
those closed subgroups of $Aut(M)$ that preserve continuous
distances on M. In particular, if M is a Kobayashi-hyperbolic
manifold, then $Aut(M)$ is a Lie group acting properly on M (see
also \cite{SK}).  Recent results concerning the dynamics of
holomorphic diffeomorphisms of compact complex surfaces are
described in \cite{SC}, \cite{SCD} and \cite{SCDP}, that require a
nice interplay between algebraic geometry, complex analysis, and
dynami- cal systems.
\medskip

In \cite{FB1} and \cite{FB2}, the author proved that if $f$ and
$g$ are two commuting holomorphic maps of the unit disc
$\triangle$, then they must share a fixed point (in the sense of
non-tangential limit if this point is on the boundary). This
result has been generalized to domains in several variables by the
author, proving that a generic characteristic of commuting
holomorphic mappings is that of sharing a "fixed point" in the
closure of the domain.

\medskip

 In this paper, we give an important condition so
that an orbit $G(x)$ is isomorphic (by linear map) to the orbit
$L_{G}(x)$ of the linear group $L_{G}$. The abelian case can be
viewed as a generalization  of the results given in \cite{aAhM05}.
We use a construction analogous to that given by S.Chihi in
\cite{Ch} for abelian linear group.

\medskip

Denote by $L_{G}=\{D_{0}f,\ f\in G\}$, it is a  subgroup of $GL(n,
\mathbb{C})$ (see Lemma ~\ref{L:008}) and $vect(L_{G})$ be the
vector space generated by $L_{G}$. For every $u\in
\mathbb{C}^{n}$, denote by:\
\\ - $L_{G}(u)=\{Au,\ A\in L_{G}\}$ the orbit of $u$ defined by
the natural action of the linear group $L_{G}$ on
$\mathbb{C}^{n}$. \
\\ - $\widetilde{E}(u)$ be the vector space generated by $L_{G}(u)$.
\ \\ -$\mathcal{A}(G)$ be the vector space generated by $G$, we
see that $\mathcal{A}(G)$ is not an algebra (i.e. it is not stable
by composition).\
\\ - $E(x)=\{f(x),\ \ f\in \mathcal{A}(G)\}$.\ \\ - For a fixed
point $x\in \mathbb{C}^{n}$, we define the linear map $\Phi_{x}:
\mathcal{A}(G)\longrightarrow \mathbb{C}^{n}$ given by
$\Phi_{x}(f)=f(x)$, $f\in \mathcal{A}(G)$.\ \\ - The linear map
$\varphi: \mathcal{A}(G)\longrightarrow M_{n}(\mathbb{C})$ the
linear map given by $\varphi(f)=D_{0}f$, for every $f\in
\mathcal{A}(G)$. \medskip

The group $G$ is called \emph{dominant} if there is a point $x\in
\mathbb{C}^{n}$ such that each set $G(x)$ and $L_{G}(x)$ generates
$\mathbb{C}^{n}$ and $ker(\varphi)= Ker(\Phi_{x})$(i.e.  $G$ is
dominant if $\widetilde{E}(x)=E(x)=\mathbb{C}^{n}$ for some $x\in
\mathbb{C}^{n}$ and $ker(\varphi)= Ker(\Phi_{x})$).

\medskip

 We generalize the
result given in \cite{aAhM05} for abelian subgroup of $GL(n,
\mathbb{C})$ in the following Theorem:
\medskip

\begin{thm}\label{T:001}  Let $G$ be a  dominant subgroup
 of  $Aut(\mathbb{C}^{n})$,  with $0\in Fix(G)$. Let $x\in
 \mathbb{C}^{n}$ such that $Ker(\Phi_{x})=Ker(\varphi)$ then $G(x)$
 is isomorphic (by a linear map) to $L_{G}(x)$.
\end{thm}
\medskip

If $G$ is abelian, we can better describe its dynamics as follow:

\begin{thm}\label{T:1}  Let $G$ be an abelian dominant subgroup
 of  $Aut(\mathbb{C}^{n})$,  with $0\in Fix(G)$.
Then there exist a $G$-invariant open set $U$, dense in
$\mathbb{C}^{n}$ satisfying: For every $y\in U$ such that
$Ker(\Phi_{y})=Ker(\varphi)$ then $G(y)$ is relatively minimal in
$U$ .
\end{thm}
\medskip

 We have the following corollaries.

\begin{cor}\label{C:4} Let $G$ be an abelian dominant subgroup of  $Aut(\mathbb{C}^{n})$,
with $0\in Fix(G)$. If $G$ has a dense orbit then every orbit in
$U$ is dense in $\mathbb{C}^{n}$.
\end{cor}

\medskip

The proofs in this paper will use the results of the action of the
linear Abelian groups on $\mathbb{C}^{n}$, that will be applied to
the linear group $\{D_{0}f,\ f \in G\}$. The complex analysis is
not used here. This paper is organized as follows: In Section 2,
we give some results for abelian linear group. The Section 3 is
devoted to prove the main results. In the section 4, we give two
examples for $n=2$.
\bigskip

\bigskip

\section{\bf Proof of Theorem~\ref{T:001}}
\
Let $G$ be a dominant subgroup of $Aut(\mathbb{C}^{n})$ of
$\mathbb{C}^{n}$ fixing $0$. Denote by:\
\\ -  $C^{1}(\mathbb{C}^{n},
\mathbb{C}^{n})$ the vector space of all $C^{1}$-differentiable
maps of $\mathbb{C}^{n}$, it is well known that
$C^{1}(\mathbb{C}^{n}, \mathbb{C}^{n})$ is an algebra.\
\\
 - $\mathcal{A}(G)$ be the vector space generated by $G$. (i.e. The
smaller vector subspace of $C^{1}(\mathbb{C}^{n}, \mathbb{C}^{n})$
containing $G$ and stable by compositions).\
\\ - For a fixed point $x\in \mathbb{C}^{n}$, we define the linear
map $\Phi_{x}: \mathcal{A}(G)\longrightarrow \mathbb{C}^{n}$ given
by $\Phi_{x}(f)=f(x)$, $f\in \mathcal{A}(G)$.\ \\ -
$E(x)=\Phi_{x}(\mathcal{A}(G))$.\ \\

\begin{lem}\label{L:2} Let $G$ be an abelian subgroup of $Aut(\mathbb{C}^{n})$, fixing $0$.
Then $g(0)=0$ for every $g\in\mathcal{A}(G)$.
\end{lem}
\medskip

\begin{proof} Let $g=\underset{k=1}{\overset{p}{\sum}}
\alpha_{k}f_{k}\subset \mathcal{A}(G)$ with $f_{k}\in G$ and
$\alpha_{k}\in \mathbb{C}$, so
$$g(0)=\underset{k=1}{\overset{p}{\sum}} \alpha_{k}f_{k}(0)=0.$$
\end{proof}
\medskip

Denote by $\varphi: \mathcal{A}(G)\longrightarrow
M_{n}(\mathbb{C})$ the linear map given by $\varphi(f)=D_{0}f$,
for every $f\in \mathcal{A}(G)$. One observes that
$L_{G}=\varphi(G)$.\medskip

\begin{lem}\label{L:008} Let $G$ be a subgroup of $Aut(\mathbb{C}^{n})$,
 fixing $0$. Then  $L_{G}$ is a  subgroup of \ $GL(n,\mathbb{C})$.
 Moreover, if $G$ is abelian so is $L_{G}$.
\end{lem}
\medskip

\begin{proof} Since $L_{G}=\varphi(G)$, so it suffices to show that $\varphi$ is a morphism. Let
$f,g\in G$, so $\varphi(f\circ g)=D(f\circ g)(0)=Df(g(0)). Dg(0)$.
By Lemma~\ref{L:2}, $g(0)=0$, so $\varphi(f\circ g)=D(f)(0).
Dg(0)=\varphi(f). \varphi(g)$. The proof is completed.
\end{proof}

\medskip

\subsection{{\bf Hamel basis and norm}}\cite{CDAKCB}
The main of this section is to justify the existence of a basis of
every vector space. This result is trivial in the finite case, is
in fact rather surprising when one thinks of infinite dimensionial
vector spaces, and the definition of a basis. Recall that a
\emph{Hamel basis} or simply a basis of a vector space $E$ is a
linearly independent set $\mathcal{B}$ (every finite subset of
$\mathcal{B}$ is linearly independant) such that for each nonzero
$x\in E$ there are $a_{1},\dots, a_{k}\in \mathcal{B}$ and nonzero
scalars $\alpha_{1},\dots,  \alpha_{k}$ (all uniquely determined)
such that $x=\underset{i=1}{\overset{k}{\sum}}\alpha_{i}a_{i}$.
 The following theorem is equivalent to the
axiom of choice family of axioms and theorems. In \cite{CDAKCB},
C.D.Aliprantis and K.C.Border  proved, in the following theorem,
that Zorn's lemma implies that every vector space has a basis.
\bigskip

\begin{thm}\label{T:aaa000a}$($\cite{CDAKCB}, Theorem 1.8$)$ Every
nontrivial vector space has a Hamel basis.
\end{thm}
\medskip

As a consequence, we found the important following results:

\begin{thm}\label{T:TTT00a} Every
nontrivial vector space has a norm called Hamel norm.
\end{thm}
\medskip

\begin{proof} Let $E$ be a nontrivial vector space over $\mathbb{C}$. By Theorem
~\ref{T:aaa000a}, $E$ has a Hamel basis called
$\mathcal{B}=(a_{i})_{i\in I}$, for any set $I$ of indices (not
necessary countable). In this basis,  every vector $x\in E$ has
the form $x=\underset{i\in I_{x}}{\sum}\alpha_{i}a_{i}$, where
$\alpha_{i}\in \mathbb{C}$ and $I_{x}\subset I$ with finite
cardinal. The family $(\alpha_{i})_{i\in I}$ with $\alpha_{i}=0$
for every $i\in I\backslash I_{x}$, is called the coordinate of
$x$. Now, define $\|x\|=\underset{i\in I_{x}}{\sum}|\alpha_{i}|$.
It is easy to verify that $\|.\|$ defines a norm on $E$ by using
the coordinate in the Hamel basis. We say that $\|.\|$ is the
Hamel norm associated to the Hamel basis $\mathcal{B}$.
\end{proof}
\medskip

Remark that any vector for the Hamel basis is with norm $1$.

\subsection{{\it Linear map and isomorphism}} A subset $E\subset\mathbb{C}^{n}$ is
called $G$-invariant if $f(E)\subset E$ for any $f\in G$; that is
$E$ is a union of orbits. For a fixed vector $x\in
\mathbb{R}^{n}\backslash\{0\}$, denote by:
 \ \\ - $F_{x}$ is an algebraic
supplement of $Ker(\Phi_{x})$ in $\mathcal{A}(G)$. It is easy to
show that $p_{x}:=dim(F_{x})=dim(E(x))\leq n$ since
$E(x)=\Phi_{x}(\mathcal{A}(G))$.\
\\ - $\mathcal{C}_{x}=(a_{1},\dots, a_{p_{x}})$ is a basis of
$F_{x}$.\ \\ - $\mathcal{B}_{x}=(b_{i})_{i\in I}$ is a Hamel basis
(Theorem~\ref{T:aaa000a}) of $Ker(\Phi_{x})$.\
\\
- $\mathcal{E}_{x}=(\mathcal{C}_{x},\mathcal{B}_{x})$ is a Hamel
basis of $\mathcal{A}(G)$. By Theorem ~\ref{T:TTT00a},
$\mathcal{A}(G)$ is provided with the Hamel norm  associated to
the basis $\mathcal{E}_{x}$.\medskip

\begin{lem}\label{L:aaaaa101001} The linear map
$\Phi_{x}:\mathcal{A}(G)\longrightarrow E(x)$ is continuous. In
particular, $Ker(\Phi_{x})$ is a closed subspace of
$\mathcal{A}(G)$.
\end{lem}
\smallskip

\begin{proof} Since $\Phi_{x}$ is linear and $\mathcal{A}(G)$ is
a normed vector space (Lemma ~\ref{T:TTT00a}), we will verify the
continuity of $\Phi_{x}$ on $0$. Let $f\in \mathcal{A}(G)$ and
write $f=f_{1}+f_{2}$ with $f_{1}\in F_{x}$ and $f_{2}\in
Ker(\Phi_{x})$. Set $(\alpha_{i})_{1\leq i\leq p_{x}}$ and
$(\beta_{i})_{i\in I}$ be respectively the coordinates of $f_{1}$
and $f_{2}$ in $\mathcal{C}_{x}$ and $\mathcal{B}_{x}$. Write
$f=\underset{i=1}{\overset{p_{x}}{\sum}}\alpha_{i}a_{i}+\underset{i\in
I_{2}}{\sum}\beta_{i}b_{i}$ where $I_{2}\subset I$ with finite
cardinal. We have $
\|f\|=\underset{i=1}{\overset{p_{x}}{\sum}}|\alpha_{i}|+\underset{i\in
I_{2}}{\sum}|\beta_{i}|$ and $b_{i}(x)=0$ for all $i\in I_{2}$.
Therefore
\begin{align*}\|\Phi_{x}(f)\|=\|f(x)\|& =\left\|\underset{i=1}{\overset{p_{x}}{\sum}}\alpha_{i}a_{i}(x)+\underset{i\in
I_{2}}{\sum}\beta_{i}b_{i}(x)\right\|\\ \ &
\leq\underset{i=1}{\overset{p_{x}}{\sum}}|\alpha_{i}|\|a_{i}(x)\|
\\ \ & \leq  \|f\|
 \underset{i=1}{\overset{p_{x}}{\sum}}\|a_{i}(x)\|\end{align*}
 Since $ \underset{i=1}{\overset{p_{x}}{\sum}}\|a_{i}(x)\|$ is
 constant relative to $f$, then $\Phi_{x}$ is continuous.
 \end{proof}
\bigskip

\begin{lem}\label{LAAL:01} Suppose that $dim(vect(L_{G}))=n$. Then the linear map
$\varphi:\mathcal{A}(G)\longrightarrow vect(L_{G})$ is continuous.
In particular, $Ker(\varphi)$ is a closed subspace of
$\mathcal{A}(G)$.
\end{lem}
\smallskip

\begin{proof} Since $\varphi$ is linear and $\mathcal{A}(G)$ is
a normed vector space (Lemma ~\ref{T:TTT00a}), we will verify the
continuity of $\varphi$ on $0$. Firstly, see that
$cod(Ker(\varphi))=n$ is finite since $dim(vect(L_{G}))=n$. let
$F$ be an algebraic supplement to $Ker(\varphi)$ in
$\mathcal{A}(G)$, $\mathcal{C}'=(a'_{1},\dots, a'_{n})$ and
$\mathcal{B}'=(b'_{i})_{i\in J}$ are respectively the Hamel basis
of $F$ and $Ker(\varphi)$ (Lemma~\ref{T:aaa000a}). Let $f\in
\mathcal{A}(G)$ and write $f=f_{1}+f_{2}$ with $f_{1}\in F$ and
$f_{2}\in Ker(\varphi)$. Set $(\alpha_{i})_{1\leq i\leq q}$ and
$(\beta_{i})_{i\in J}$ be respectively the coordinates of $f_{1}$
and $f_{2}$ in $\mathcal{C}'$ and $\mathcal{B}'$. Write
$f=\underset{i=1}{\overset{n}{\sum}}\alpha_{i}a'_{i}+\underset{i\in
I_{2}}{\sum}\beta_{i}b'_{i}$ where $I_{2}\subset j$ with finite
cardinal. We have $
\|f\|=\underset{i=1}{\overset{n}{\sum}}|\alpha_{i}|+\underset{i\in
I_{2}}{\sum}|\beta_{i}|$ and $\varphi(b'_{i})=Db'_{i}(0)=0$ for
all $i\in I_{2}$. Therefore
\begin{align*}\|\varphi(f)\|=\|Df(0)\|& =\left\|\underset{i=1}{\overset{n}{\sum}}\alpha_{i}Da'_{i}(0)+\underset{i\in
I_{2}}{\sum}\beta_{i}Db'_{i}(0)\right\|\\ \ &
\leq\underset{i=1}{\overset{n}{\sum}}|\alpha_{i}|\|Da'_{i}(0)\| \\
\ & \leq  \|f\|
 \underset{i=1}{\overset{n}{\sum}}\|Da'_{i}(0)\|\end{align*}
 Since $ \underset{i=1}{\overset{n}{\sum}}\|Da'_{i}(0)\|$ is
 constant relative to $f$, then $\varphi$ is continuous.
 \end{proof}
\bigskip

\begin{lem}\label{L:GTM1}$($\cite{GTM}, $3.5)$ Let $E$ be a topological vector
space, and let $M$ be a closed subspace of finite codimension.
Then $E=M\oplus N$ is a topological sum, for every algebraic
complementary subspace $N$ of $M$.
\end{lem}
\medskip

\begin{cor}\label{CLC:1} The algebraic sum $\mathcal{A}(G)=F_{x}\oplus
Ker(\Phi_{x})$ is topological. In particular, $F_{x}\oplus
Ker(\Phi_{x})$  and $\mathcal{A}(G)$ are topological isomorphic by
the map: $(f_{1},f_{2})\longmapsto f_{1}+f_{2}$.
\end{cor}
\medskip

\begin{proof} By the Theorem~\ref{T:TTT00a}, $\mathcal{A}(G)$ is a
normed vector space so it is a topological vector space. By Lemma
~\ref{L:aaaaa101001}, $\Phi_{x}$ is continuous and so its kernel
is closed vector space with finite codimension. The results
follows directly by applying the Lemma~\ref{L:GTM1} for
$E=\mathcal{A}(G)$ and $M=Ker(\Phi_{x})$.
\end{proof}
\medskip
\
\\
By Corollary ~\ref{CLC:1}, we can identify $\mathcal{A}(G)$ with
$F_{x}\oplus Ker(\Phi_{x})$, so every $f\in \mathcal{A}(G)$ is
denoted by $f=(f_{1},f_{2})=f_{1}+f_{2}$ with $f_{1}\in F_{x}$ and
$f_{2}\in Ker(\Phi_{x})$.\medskip

\begin{lem}\label{LL:aaa010+} Let $H$ and $K$ be two closed vector subspaces of $\mathcal{A}(G)$
 such that $cod(H)=cod(K)=j\geq 1$. Let $\psi\in \mathcal{A}(G)\backslash (H\cup K)$. Then
there exists a vector space $F$ of $\mathcal{A}(G)$, containing
$\psi$ and satisfying $\mathcal{A}(G)=F\oplus H=F\oplus K$.
\end{lem}

\medskip

\begin{proof} There are two cases;\ \\ \emph{Case 1:} If $j=1$, then we take $F=\mathbb{C}\psi$.\ \\ \emph{Case 2:} Suppose that $j\geq 2$.
Denote by $H'=H\oplus\mathbb{C}\psi$ and
$K'=K\oplus\mathbb{C}\psi$. Suppose that $H'\neq K'$ (otherwise,
it is easy to take a comment supplement) and let $f_{1}\in
H'\backslash K'$ and $g_{1}\in K'\backslash H'$, so
$h_{1}=f_{1}+g_{1}\notin H'\cup K'$. Denote by
$H_{1}=\mathbb{C}h_{1}\oplus H'$ and $K_{1}=\mathbb{C}h_{1}\oplus
K'$. We establish two cases:\
\\ - If $H_{1}=K_{1}$, then any supplement $F$ of $H_{1}$ in
$\mathcal{A}(G)$ is a supplement of $K_{1}$ in $\mathcal{A}(G)$,
the proof follows then.\ \\ - If $H_{1}\neq K_{1}$, we take
$f_{2}\in H_{1}\backslash K_{1}$ and $g_{2}\in K_{1}\backslash
H_{1}$, so $h_{2}=f_{2}+g_{2}\notin H_{1}\cup K_{1}$. Denote by
$H_{2}=\mathbb{C}h_{2}\oplus H_{1}$ and
$K_{2}=\mathbb{C}h_{2}\oplus K_{1}$.
\medskip

  We repeat the same processes until $j-2$ times and we obtain:\ \\ - If
$H_{j-2}=K_{j-2}$, then any supplement $F$ of $H_{j-2}$ in
$\mathcal{A}(G)$ is a supplement of $K_{j-2}$ in
$\mathcal{A}(G)$.\
\\ - If $H_{j-2}\neq K_{j-2}$, we take $f_{j-1}\in H_{j-2}\backslash
K_{j-2}$ and $g_{j-1}\in K_{j-2}\backslash H_{j-2}$, so
$h_{j-1}=f_{j-1}+g_{j-1}\notin H_{j-2}\cup K_{j-2}$. Denote by
$H_{j-1}=\mathbb{C}h_{j-1}\oplus H_{j-2}$ and
$K_{j-1}=\mathbb{C}h_{j-1}\oplus K_{j-2}$. We obtain then
$H_{j-1}=K_{j-1}=\mathcal{A}(G)$. Hence the proof is completed by
taking $F=vect(\psi,h_{1},\dots, h_{j-1})$.
\end{proof}
\medskip
\ \\
  Denote by
$\widetilde{r}_{x}=dim(\widetilde{E}(x))$. \medskip

\begin{lem}\label{LIL:000}  If $r(x)=\widetilde{r}_{x}$, then there exists a
commune vector space $F_{x}$ supplement to $Ker(\Phi_{x})$ and to
$ker(\varphi)$ in $\mathcal{A}(G)$. (i.e $F_{x}\oplus
Ker(\Phi_{x})=F_{x}\oplus Ker(\varphi)=\mathcal{A}(G)$).
\end{lem}
\medskip

\begin{proof} By Lemmas ~\ref{L:aaaaa101001} and ~\ref{LAAL:01}, the maps $\Phi_{x}$ and $\varphi$ are
continuous, so  $Ker(\Phi_{x})$ and $Ker(\varphi)$ are closed.
Since $cod(Ker(\Phi_{x}))=cod(Ker(\varphi))=n$, then by
lemma~\ref{LL:aaa010+}, there exists a commune supplement $F_{x}$
to $Ker(\Phi_{x})$ and to $ker(\Phi_{x})$ in $\mathcal{A}(G)$
containing $id$ (identity map of $\mathbb{C}^{n}$), because $id\in
\mathcal{A}(G)\backslash (Ker(\Phi_{x})\cup Ker(\varphi))$.
\end{proof}
\medskip

\begin{lem}\label{LDL:000} The linear map $\Phi_{x}: F_{x}\longrightarrow E(x)$
given by $\Phi_{x}(f)=f(x)$ is an isomorphism.
\end{lem}
\medskip

\begin{proof} Here, $F_{x}$ is considered as a supplement to $Ker(\Phi_{x})$ in $\mathcal{A}(G)$. The proof follows directly from the fact that
$\Phi_{x}$ is linear surjective and $dim(F_{x})=dim(E(x))=n$.
\end{proof}
\medskip

\begin{lem}\label{LVL:000}  If $r(x)=\widetilde{r}_{x}=n$, then the restriction
 $\varphi_{1}: F_{x}\longrightarrow vect(L_{G})$
 of $\varphi$ from $F_{x}$ unto $vect(L_{G})$ is an isomorphism.
\end{lem}
\medskip

\begin{proof} Here, $F_{x}$ is considered as a supplement to $Ker(\varphi)$ in $\mathcal{A}(G)$. The proof follows directly from the fact that $\varphi_{1}$ is linear surjective and
$dim(F_{x})=dim(vect(L_{G}))=n$ (Corollary ~\ref{CCC:1001}).
\end{proof}
\medskip

\begin{lem}\label{LL+L:000} If $r(x)=\widetilde{r}_{x}=n$ and $ker(\varphi)= Ker(\Phi_{x})$, then  the map
 $\varphi_{x}:=\widetilde{\Phi}_{x}\circ \varphi_{1}
\circ \Phi_{x}^{-1}: \mathbb{C}^{n}\longrightarrow \mathbb{C}^{n}$
defined by $\varphi_{x}(f(x))=D_{0}f(x)$, $f\in F_{x}$, is an
isomorphism and satisfying:\ \\ (i) $\varphi_{x}(G(x))= L_{G}(x)$.
\
\\ (ii) Let $y\in \mathbb{C}^{n}$ and $z=\varphi_{x}(y)$ then
$\varphi_{x}(G(y))= L_{G}(z)$.
 \ \\ (iii) for every $y\in \overline{G(x)}$ we have  $z=\varphi_{x}(y)\in
\overline{L_{G}(x)}$.
\end{lem}
\medskip

\begin{proof} By Lemma ~\ref{LIL:000}, we can assume that $F_{x}$ is a commune
supplement to $Ker(\Phi_{x})$ and to $Ker(\varphi)$ in
$\mathcal{A}(G)$. Since $r(x)=\widetilde{r}_{x}=n$, then
$E(x)=\widetilde{E}(x)=\mathbb{C}^{n}$.\
\\ (i) $\varphi_{x}$ is well defined; Indeed, let $f\in
\mathcal{A}(G)$ and write $f=g'+h$ with $g'\in Ker(\Phi_{x})$ and
$h\in Ker(\varphi)$, so $f(x)=g'(x)$. Then
$\varphi_{x}(f(x))=\widetilde{\Phi}_{x}\circ \varphi_{1} \circ
\Phi_{x}^{-1}(g'(x))=\widetilde{\Phi}_{x}\circ
\varphi_{1}(g')=\widetilde{\Phi}_{x}(D_{0}g')=D_{0}g'(x)$. Since
$Ker(\Phi_{x})=Ker(\varphi)$ then $D_{0}h=0$, so $D_{0}f=D_{0}g'$.
It follows that $\varphi_{x}(f(x))=D_{0}f(x)$.\ \\ $\varphi_{x}$
is an isomorphisms, since $\varphi_{x}=\widetilde{\Phi}_{x}\circ
\varphi_{1} \circ \Phi_{x}^{-1}$ and  by Lemmas ~\ref{LVL:000},
~\ref{L:0000a0a0} and ~\ref{LDL:000}, it is composed by
isomorphisms.

Let $y\in \varphi_{x}(G(x))$, there exists $f \in G$ such that
$y=\varphi_{x}(f(x))$, then
$\varphi_{x}(y)=\varphi_{x}(f(x))=D_{0}f(x)\in L_{G}(x)$. For the
converse, let $y\in L_{G}(x)$, there exists $f \in G$ such that
$y=D_{0}f(x)$. Write $f=g+h$ with $g\in F_{x}$ and $h\in
Ker(\varphi)$. Therefore
$y=D_{0}g(x)+D_{0}h(x)=D_{0}g(x)=\varphi_{x}(g(x))$.  Since
$Ker(\Phi_{x})=Ker(\varphi)$, then $h\in Ker(\Phi_{x})$, so
$h(x)=0$. As $f(x)=g(x)+h(x)=g(x)$, it follows that
$\varphi_{x}(f(x))=y$. Hence, $y\in\varphi_{x}(G(x))$.

\ \\ \\ (ii) Let $y\in \mathbb{C}^{n}$ and $z=\varphi_{x}(y)$.  By
Lemma ~\ref{LDL:000}, $\Phi_{x}(F_{x})=E(x)=\mathbb{C}^{n}$, then
there exists $g\in F_{x}$ such that $y=g(x)$, so $z=D_{0}g(x)$.
Let $f\in G$ then $\varphi_{x}(f(y))=D_{0}(f\circ g)
(x)=D_{0}f\circ D_{0}g(x)=D_{0}f(z)$. Hence, $\varphi_{x}(f(y))\in
L_{G}(z)$.\ \\ For the converse, let $a\in L_{G}(z)$, there exists
$f \in G$ such that $a=D_{0}f(z)$. Then
$\varphi_{x}(f(y))=\varphi_{x}(f\circ g(x))=D_{0}(f\circ
g)(x)=D_{0}f\circ D_{0}g(x)=D_{0}f (z)=a$. Hence,
$a\in\varphi_{x}(G(y))$.

\
\\  \\ (iii) Since $y\in \mathbb{C}^{n}$, there exists $g\in F_{x}$ such that
$y=g(x)$, so $z=D_{0}g(x)$. By continuity  of $\varphi_{x}$ and by
(i) , we have $z\in \varphi_{x}(\overline{G(x)})\subset
\overline{\varphi_{x}(G(x))}=\overline{L_{G}(x)}$.
\end{proof}

\bigskip

\section{\bf The abelian case}
\bigskip

\subsection{\bf Some results for abelian linear group}
\bigskip

Let $M_{n}(\mathbb{C})$ be the set of complex square matrices of
order $n \geq 1$, and let $GL(n,\mathbb{C})$ be the group of the
invertible matrices of $M_{n}(\mathbb{C})$. Denote by\ \\ -
$\mathbb{T}_{n}(\mathbb{C})$ the set of all lower-triangular
matrices over $\mathbb{C}$, of order $n$ and with only one
eigenvalue.\ \\ - $\mathbb{T}_{n}^{*}(\mathbb{C}) =
\mathbb{T}_{n}(\mathbb{C})\cap GL(n, \mathbb{C})$ (\textit{i.e.}
the subset of matrix of $\mathbb{T}_{n}(\mathbb{C})$ having a non
zero eigenvalue), it is a subgroup of $GL(n, \mathbb{C})$.\ \\ -
$\mathbb{C}^{*}= \mathbb{C}\backslash\{0\}$ and $\mathbb{N}_{0}=
\mathbb{N}\backslash\{0\}$. \ \\ \\ Let $r\in \mathbb{N}^{*}$ and
$\eta=(n_{1},\dots,n_{r})\in\mathbb{N}_{0}^{r}$ such that
$\underset{i=1}{\overset{r}{\sum}}n_{i}=n$. Denote by: \ \\ -
$\mathcal{K}_{\eta,r}(\mathbb{C}) = \left\{M=\mathrm{diag}(
T_{1},\dots, T_{r})\in M_{n}(\mathbb{C}): \ T_{k}\in
\mathbb{T}_{n_{k}}(\mathbb{C}),\ k=1,\dots,r \right\}.$ \ \\ -
$\mathcal{K}_{\eta,r}^{*}(\mathbb{C})=\mathcal{K}_{\eta,r}(\mathbb{C})\cap
GL(n, \mathbb{C})$, it is a subgroup of $GL(n, \mathbb{C})$. \ \\
- $v^{T}$ the transpose of a vector $v\in \mathbb{C}^{n}$. \ \\ -
$\mathcal{E}_{n}=(e_{1},\dots,e_{n})$ the standard basis of
 $\mathbb{C}^{n}$.
\ \\ -$I_{n}$ the identity matrix on $\mathbb{C}^n$.
 \ \\ - $u_{0}=[e_{1,1},\dots, e_{r,1}]^{T}\in
\mathbb{C}^{n}$, where $e_{k,1}=[1,0,\dots,0]^{T} \in
\mathbb{C}^{n_{k}}, \quad 1\leq k \leq r.$

For any subset $E$ of $\mathbb{C}^{n}$ (resp.
$M_{n}(\mathbb{C})$), denote by $vect(E)$ the vector space
generated by $E$.
\medskip

\ \\ In \cite{aAhM06}, the authors proved the following
Proposition:

\begin{prop}\label{p:001}$($\cite{aAhM06}, Proposition 6.1.$)$  Let $L$ be an abelian subgroup of $GL(n, \mathbb{C})$,
 then there exists $P\in GL(n, \mathbb{C})$ such that $\widetilde{L}=P^{-1}LP$ is a subgroup of $\mathcal{K}^{*}_{\eta,r}(\mathbb{C})$,
 for some $1\leq r\leq n$ and $\eta\in \mathbb{N}_{0}^{r}$.
\end{prop}
\ \\ \\ For such matrix $P$ define $v_{0}=Pu_{0}$. Let $L$ be an
abelian subgroup of $\mathcal{K}^{*}_{\eta,r}(\mathbb{C})$. denote
by:\ \\ -
$V=\underset{k=1}{\overset{r}{\prod}}\mathbb{C}^{*}\times
\mathbb{C}^{n_{k}-1}$. One has $\mathbb{C}^{n}\backslash
V=\underset{k=1}{\overset{r}{\bigcup}}H_{k}$, where
$$H_{k}=\left\{u=[u_{1},\dots,u_{r}]^{T},\ \ u_{k}\in \{0\}\times
\mathbb{C}^{n_{k}-1},\ u_{j}\in\mathbb{C}^{n_{j}},\ j\neq k
\right\}.$$ See that each $H_{k}$ is a $L$-invariant vector space
of dimension $n-1$.
\medskip

\begin{lem}\label{lmm:01}$($\cite{Ch}, Proposition 3.1$)$ Let $L$ be an abelian subgroup of $GL(n,
\mathbb{C})$ and $u\in \mathbb{C}^{n}$. Then for every $v\in
vect(L(u))$ there exist $B\in vect(L)$ such that $Bu =v$.
\end{lem}
\bigskip

For an abelian subgroup $L$ of $L(n, \mathbb{C})$, it is called
dominant if $vect(L(x))=\mathbb{C}^{n}$ for some $x\in
\mathbb{C}^{n}$, where $L(x)=\{Ax, \ A\in G\}$.
\medskip

\begin{prop}\label{L:0100} Let $\widetilde{L}$ be an abelian linear subgroup of
$\mathcal{K}^{*}_{\eta,r}(\mathbb{C})$ with $\eta=(n_{1},\dots,
n_{r})$. Then the following assertions are equivalent:\ \\ (i)
$\widetilde{L}$ is dominant.\ \\ (ii) For every $u\in V$, we have
$vect(\widetilde{L}(u))=\mathbb{C}^{n}$. In particular,
$vect(\widetilde{L}(u_{0}))=\mathbb{C}^{n}$.\
\end{prop}
\medskip

\begin{proof} Suppose that $\widetilde{L}$ is dominant, then there is $u\in
\mathbb{C}^{n}$ such that $vect(\widetilde{L}(v))=\mathbb{C}^{n}$.
Remark that $u\in V$, since $\mathbb{C}^{n}\backslash V$ is a
union of $r$ $\widetilde{L}$-invariant vector spaces with
dimensions $n-1$ and let $v\in V$. By applying lemma~\ref{lmm:01}
on $\widetilde{L}$, there exist $B \in vect(\widetilde{L})$ such
that $Bu = v$. As $\mathcal{K}_{\eta,r}(\mathbb{C})$ is a vector
space then $Vect(L)\subset \mathcal{K}_{\eta,r}(\mathbb{C})$.
Write $u=[u_{1},\dots,u_{r}]^{T}$, $v=[v_{1},\dots,v_{r}]^{T}$
with $u_{k}=[x_{k,1},\dots, x_{k,n_{k}}]^{T},\
v_{k}=[y_{k,1},\dots,
y_{k,n_{k}}]^{T}\in\mathbb{C}^{*}\times\mathbb{C}^{n_{k}-1}$ and
$B=diag(B_{1},\dots, B_{r})$ with
$$B_{k}=\left[\begin{array}{cccc}
  \mu_{B_{k}} & \ & \ & 0 \\
  a^{(k)}_{2,1} & \ddots & \ & \ \\
  \vdots & \ddots & \ddots & \ \\
  a^{(k)}_{n_{k},1} & \dots & a^{(k)}_{n_{k},n_{k}-1} & \mu_{B_{k}}
\end{array}\right], \ \ 1\leq k \leq r$$ then $\mu_{B_{k}}x_{k,1}=y_{k,1}$, so $\mu_{B_{k}}=\frac{y_{k,1}}{x_{k,1}}\neq 0$, hence $B\in GL(n,
\mathbb{C})$. Then \ $B(\widetilde{L}(u))=\widetilde{L}(v)$. We
conclude that $vect(\widetilde{L}(v))=\mathbb{C}^{n}$. The
converse is obvious.
\end{proof}
\medskip

 Denote by:\ \\
- $\Psi_{x}: vect(L_{G}) \longrightarrow
\widetilde{E}(x)\subset\mathbb{C}^{n}$ the linear map given by
$\Psi_{x}(A)=Ax$.\
\\

\begin{lem}\label{L:0000a0a0} Let $G$ be an abelian  subgroup of $Aut(\mathbb{C}^{n})$ such that
 $0\in Fix(G)$. Then
 $\widetilde{E}(x)$ is $L_{G}$-invariant and
the linear map $\Psi_{x}: vect((L_{G})_{/\widetilde{E}(x)})
\longrightarrow \widetilde{E}(x)$ is an isomorphism, where
$(L_{G})_{/\widetilde{E}(x)}$ is the restriction of $L_{G}$ on
$\widetilde{E}(x)$.
\end{lem}
\ \\
\begin{proof} By construction, $\widetilde{E}(x)$ is $L_{G}$-invariant and $\Psi_{x}$ is linear and
surjective. Let $A\in Ker(\Psi_{x})$ and $y\in \widetilde{E}(x)$.
Then there is $B\in vect((L_{G})_{/\widetilde{E}(x)})$ such that
$y=Bx$. Now, $Ay=ABx=BAx=0$, so $A=0$. Hence $\Psi_{x}$ is
injective, so it is an isomorphism.
\end{proof}
\medskip

\begin{cor}\label{CCC:1001} Let $G$ be an abelian dominant subgroup of $Aut(\mathbb{C}^{n})$ such that
 $0\in Fix(G)$. Then
 $dim(vect(L_{G}))=n$.
\end{cor}

\begin{proof} Since $G$ is dominant, then there is $x\in
\mathbb{C}^{n}$ such that $\widetilde{E}(x)=\mathbb{C}^{n}$. Then
by lemma~\ref{L:0000a0a0},  $\Psi_{x}: vect(L_{G}) \longrightarrow
\mathbb{C}^{n}$ is an isomorphism, so  $dim(vect(L_{G}))=n$.
\end{proof}
\medskip

Denote by: \ \\ - $\widetilde{\Omega}_{n}=\{x\in \mathbb{C}^{n},\
dim(vect(\widetilde{L}(x)))=n\}$.

\medskip

\begin{lem}\label{LDLL:01} Let $\widetilde{L}$ be a dominant abelian  subgroup of
$\mathcal{K}^{*}_{\eta,r}(\mathbb{C})$. Then
$\widetilde{\Omega}_{n}= V$.
\end{lem}
\ \\
\begin{proof} Since $\widetilde{L}$ is dominant, then by Proposition~\ref{L:0100},
 for every $u\in V$, $vect(\widetilde{L}(u))=\mathbb{C}^{n}$, hence
$V\subset \widetilde{\Omega}_{n}$. For the converse, let $u\in
\widetilde{\Omega}_{n}$, then $dim(vect(\widetilde{L}(u)))=n$. It
follows that $u\in V$ because $\mathbb{C}^{n}\backslash V$ is a
union of $r$ $\widetilde{L}$-invariant vector spaces of dimension
$n-1$. This completes the proof.
\end{proof}
\medskip

Denote by:\ \\ - $\widetilde{\Omega}_{k}=\{y\in \mathbb{C}^{n},\ \
\widetilde{r}_{y}\geq k\}$, for every $0 \leq k\leq n$. By
applying Lemma 3.11 given in \cite{Ch} to the abelian linear group
$L_{G}$, we found the following result:

\begin{lem}\label{L:10}$($\cite{Ch}, Lemma3.11$)$  $\widetilde{\Omega}_{k}$ is a
$L_{G}$-invariant dense open subset of $\mathbb{C}^{n}$.
\end{lem}
\
\\

\begin{lem}\label{L:aaaa1}$($\cite{Ch}, Theorem 3.10$)$ Let $x\in \widetilde{\Omega}_{n}$ then for every $y\in
\overline{L_{G}(x)}\cap\widetilde{\Omega}_{n}$ we have
$\overline{L_{G}(y)}\cap\widetilde{\Omega}_{n}=\overline{L_{G}(x)}\cap\widetilde{\Omega}_{n}$.
\end{lem}
\medskip

Denote by:\ \\ - $\Omega_{n}=\{y\in \mathbb{C}^{n},\ \
dim(E(y))=n\}$.
\medskip

\begin{lem}\label{L:LLLL} $($Under the above notations$)$ Suppose that $r(x)=\widetilde{r}_{x}=n$
 and $ker(\varphi)= Ker(\Phi_{x})$. We have $\varphi_{x}(\Omega_{n})=\widetilde{\Omega}_{n}$.
\end{lem}

\begin{proof} Let $y\in \Omega_{n}$ and $z=\varphi_{x}(y)$. By
Lemma~\ref{LL+L:000},(ii), $\varphi_{x}(G(y))= L_{G}(z)$. Since
$\varphi_{x}$ is linear, then $\varphi_{x}(vect(G(y))=
\widetilde{E}(z)=\mathbb{C}^{n}$, since $\widetilde{r}_{z}=n$,
where $vect(G(y))$ is the vector space generated by $G(y)$. As
$vect(G(y))\subset E(y)$, then $\varphi_{x}(E(y))=\mathbb{C}^{n}$,
so  $r(y)=n$. It follows that
 $z\in \widetilde{\Omega}_{n}$. For the converse we use the same proof for
$\varphi^{-1}_{x}$.
\end{proof}
\medskip

\subsection{{\bf Proof of Theorem ~\ref{T:1}}}

\ \\ Denote by:\ \\ - $r(x)= dim(E(x))$.\ \\ - $U_{k}=\{x\in
\mathbb{C}^{n}, \ r(x)\geq k\}$, for every $k\in \mathbb{N}$.\ \\
- $r_{G}=max\{r(x),\ \ x\in \mathbb{C}^{n} \}$.\medskip

\begin{prop}\label{p:100} Let $G$ be an abelian subgroup of
 $Aut(\mathbb{C}^{n})$, such that $0\in Fix(G)$.
Then for every $0\leq k\leq r_{G}$, $U_{k}$ is a $G$-invariant
open subset of $\mathbb{C}^{n}$.
\end{prop}
\
\\
\begin{proof} In the first, remark that the rank $r(y)$ is constant on any orbit $G(y)$,
 $y\in E(x)$. So $U_{k}$ is
$G$-invariant for  every $0\leq k \leq r_{G}$. Let's show that
$U_{t}$ is an open set: Let $y\in U_{t}$ and $r=r_{y}$, so $r\geq
t$. Then there exist $f_{1},\dots, f_{r}\in G$ such that the $r$
vectors $f_{1}(y),\dots, f_{r}(y)$ are linearly independent in
$E(y)$. For all $z\in \mathbb{C}^{n}$, we consider the Gram's
determinant $$\Delta(z)=det\left(\langle f_{i}(z)\ |\
f_{j}(z)\rangle\right)_{1\leq i,j\leq r}$$ of the vectors
$f_{1}(z),\dots, f_{r}(z)$ where $\langle. | . \rangle$ denotes
the scalar product in $\mathbb{C}^{n}$. It is well known that
these vectors are independent if and only if $\Delta(z) \neq 0$,
in particular $\Delta(y)\neq 0$. Let $$V_{y}=\left\{z\in
\mathbb{C}^{n},\ \ \Delta(z) \neq 0\right\}$$ The set $V_{y}$
  is open in $\mathbb{C}^{n}$, because the map $z\longmapsto \Delta(z)$ is continuous.
   Now $\Delta(y)\neq 0$, and so
   $y\in V_{y}\subset U_{k}$.\
     The proof is completed.
\end{proof}

\begin{proof}[Proof of Theorem~\ref{T:1}]
Since $G$ is dominant, then there is $a\in \mathbb{C}^{n}$ such
that $r(a)=\widetilde{r}_{a}=n$ and $ker(\varphi)= Ker(\Phi_{a})$,
so $\Omega_{n}\neq\emptyset$. Let $U=\Omega_{n}\cap
\widetilde{\Omega}_{n}$, since
$\overline{\Omega_{n}}=\overline{\widetilde{\Omega}_{n}}=\mathbb{C}^{n}$,
then $\overline{U}=\mathbb{C}^{n}$. Let $x\in U$ such that
$Ker(\varphi)=Ker(\Phi_{x})$ and let $y\in \overline{G(x)}\cap U$
such that. By Lemma~\ref{LL+L:000}, there exists an isomorphism
$\varphi_{x}: \mathbb{C}^{n}\longrightarrow \mathbb{C}^{n}$
satisfying $\varphi_{x}(G(x))=L_{G}(x)$ and
$\varphi_{x}(G(y))=L_{G}(z)$ with $z=\varphi_{x}(y)$.  By
Lemma~\ref{L:LLLL}, we have
$\varphi_{x}(\Omega_{n})=\widetilde{\Omega}_{n}$, so $z\in
\varphi_{x}(\overline{G(x)}\cap
\Omega_{n})=\overline{L_{G}(x)}\cap \widetilde{\Omega}_{n}$. By
Lemma ~\ref{L:aaaa1}, we have
$$\overline{L_{G}(z)}\cap\widetilde{\Omega}_{n}=\overline{L_{G}(x)}\cap\widetilde{\Omega}_{n}\
\ \ \ \ (1).$$  Therefore by (1) we obtain  \begin{align*}
\overline{G(x)}\cap\Omega_{n}&=\varphi_{x}^{-1}(\overline{L_{G}(x)})\cap\varphi_{x}^{-1}(\widetilde{\Omega}_{n})\\
\ & =
\varphi_{x}^{-1}(\overline{L_{G}(x)}\cap\widetilde{\Omega}_{n})
\\
\ &
=\varphi_{x}^{-1}(\overline{L_{G}(z)})\cap\widetilde{\Omega}_{n})\\
\ &
=\varphi_{x}^{-1}(\overline{L_{G}(z)})\cap\varphi_{x}^{-1}(\widetilde{\Omega}_{n})\\
 \ & = \overline{G(y)}\cap\Omega_{n}.
\end{align*}

We conclude that $\overline{G(x)}\cap U=\overline{G(y)}\cap U$.
 This completes the proof.
 \end{proof}
\medskip

\begin{proof}[Proof of Corollary~\ref{C:4}]
 Let $O$ be a dense orbit in $\mathbb{C}^{n}$ (i.e. $\overline{O} =\mathbb{C}^{n}$).
Then for every $x\in O$, we have $r(x)=n$, so $O\subset U$ and
$\overline{O}\cap U = U$. Since $O$ is relatively minimal in $U$
(Theorem ~\ref{T:1}), then for every orbit $L\subset U$,we have
$\overline{L}\cap U =\overline{O} \cap U$. Therefore
$\overline{L}=\overline{O} =\mathbb{C}^{n}$.
\end{proof}
\medskip

\medskip

\section{{\bf Particular cases }}
\subsection{{\bf Linear case}} Suppose that $G$ is
a dominant abelian subgroup of $GL(n, \mathbb{C})$. In this case,
we have $G=L_{G}$ and $\mathcal{A}(G)=vect(G)$. It follows that
$\mathcal{A}(G)$ has a  finite dimension. Moreover,
$G(x)=L_{G}(x)$, so the theorem\ref{T:001} becomes trivial. For
the theorem\ref{T:1}, we found the structure's theorem given in
\cite{aAhM05}, but we give an other form of the $G$-invariant open
set $U$ as follow: $$U:=\{x\in \mathbb{C}^{n},\ \ A\in vect(G)\
\mathrm{with }\ Ax=0\ \Longrightarrow\ \ A=0\}.$$ This means that
$$U:=\{x\in \mathbb{C}^{n},\ \ Ker(\Phi_{x})=\{0\}\}.$$

\medskip

\subsection{{\bf Affine case}} Suppose that $G$ is
a dominant abelian group of affine maps of $\mathbb{C}^{n}$ fixing
a commune point. Denote by $Fix(G)$ the set of all commune fixed
points of all elements of $G$, so $Fix(G)\neq \emptyset$. In this
case, every $f\in G$ has the form: $f(x)=Ax+b$ with $A\in GL(n,
\mathbb{C})$ and $b\in \mathbb{C}^{n}$. Write $f:=(A, b)$ and we
have $L_{G}=\{A,\ \ (A, b)\in G\}$ (i.e. $L_G$ is the group of
linear part of all $f\in G$). In \cite{AY}, (Proposition 3.1), the
authors proved that $G$ is $\mathcal{T}$-conjugate (conjugation by
translation) to $L_{G}$, so there exists a translation $T$ of
$\mathbb{C}^{n}$ such that $G(x)=T(L_G(T^{-1}x))$. But by
theorem\ref{T:001}, we give a linear isomorphism $\varphi_x$
satisfying $G(x)=\varphi_x(L_G(x))$. Here,
$\mathcal{A}(G)=vect(G)$. It follows that $\mathcal{A}(G)$ has a
finite dimension. For the theorem\ref{T:1}, we found the
structure's theorem given in \cite{AY}, but we give an other form
of the $G$-invariant open set $U$ as follow: $$U:=\{x\in
\mathbb{C}^{n},\ \ (A,b)\in vect(G)\ \mathrm{with }\ Ax+b=0\
\Longrightarrow\ \ A=0\}.$$ This means that $$U:=\{x\in
\mathbb{C}^{n},\ \ Ker(\Phi_{x})=\{0\}\}.$$

\medskip

\bibliographystyle{amsplain}
\vskip 0,4 cm

\end{document}